\documentclass[12pt]{amsart}%
\usepackage{amsmath,amsthm,amsfonts,amssymb}
\usepackage{verbatim}
\usepackage{latexsym}
\usepackage{marginnote}
\usepackage{amscd}
\usepackage{amssymb}
\usepackage{xcolor}
\usepackage{amsmath}
\usepackage{amsfonts}
\usepackage{mathrsfs}
\usepackage{graphicx}%
\providecommand{\U}[1]{\protect\rule{.1in}{.1in}}
\providecommand{\U}[1]{\protect\rule{.1in}{.1in}}
\providecommand{\U}[1]{\protect\rule{.1in}{.1in}}
\newtheorem{theorem}{Theorem}[section]

\newtheorem{definition}[theorem]{Definition}

\newtheorem{remark}[theorem]{Remark}

\theoremstyle{definition}

\begin{document}
\title{ Weighted Gelfand pairs, weighted spherical Fourier transform and multipliers}
\author{Ass\`ek\`e Y. Tissinam, Abudula\"{i} Issa  and  Yaogan Mensah}

\begin{abstract}
A kind of generalized Gelfand pair is introduced via a Banach algebra consisting of bi-invariant functions in a weighted 
Lebesgue space. The related spherical functions and the spherical Fourier transformation are constructed. The multipliers of the underlying algebra are characterized by this Fourier transformation.
\end{abstract}
\maketitle


\section{Introduction}

Let $G$ be a locally compact group and $K$ be a compact subgroup of $G$. The concept of Gelfand pair is linked to the commutativity of the space of compact supported $K$-bi-invariant functions on $G$ under the convolution product \cite{Dieudonne,Faraut}. It allowed to go beyond the harmonic analysis on locally compact abelian groups and paved the way for the harmonic analysis on a more general commutative spaces \cite{Dijk, Wolf}.  The theory of Gelfand pairs has  many   applications in some fields of mathematical sciences such as harmonic analysis on Riemannian symmetric spaces, number theory and coding theory \cite{Gross}.

The theory of Gelfand pairs, and consequently that of spherical functions,  has undergone several generalizations \cite{Campos, Gallo, Toure}. 

In this article, we sought to generalize the notion of Gelfand pair in the direction of weighted Lebesgue spaces \cite{Kinani}  using a convolution product which depends on the weight. Mainly, we study weighted Gelfand pairs, weighted spherical functions, weighted spherical Fourier transform  and multipliers. 

The notion of multiplier originates from the  theory of summability for Fourier series \cite{Larsen}. It extends to other areas  with  many applications such  as  probability theory, optimization theory, differential equations, theory of signal, acoustics, etc. Some pioneer works in this area are \cite{Helgason, Wang, Wendel 1,Wendel 2}. In 2020, authors in \cite{Issa 1} obtained a theorem of Wendel type related to multipliers for Gelfand pairs.

The rest of the paper is organized  as follows. Section \ref{Preliminaries} is devoted to the recall of some basic facts about the weighted group algebra and Gelfand pairs. In Section \ref{Mainresults}, we state the main results  spread over three subsections.  

\section{Preliminaries}\label{Preliminaries}
\subsection{The weighted group algebra}
Let $G$ be a locally compact group which neutral element is denoted $e$.  Let $K$ be a compact subgroup of $G$.  
In the whole paper, integration of functions on $G$ is taken with respect to a left Haar measure of $G$ while  integration of functions on $K$ is with respect with the normalized Haar measure of $K$. 

On the Lebesgue space 
 $$ L^{1}(G)= \left\{f: G \rightarrow \mathbb{C}: \displaystyle\int_{G}^{}|f(x)|dx < \infty\right\}$$
one defines the norm 
 $$ \lVert f \rVert_{1}=\displaystyle\int_{G} \lvert f(x) \rvert dx$$
  and  the   convolution product  
  
   $$f \ast  g(x) = \int_{G}^{}f(y)g(y^{-1}x)dy.$$
    
  In this article, we refer to this convolution product as the classical convolution product. It is well-known that $\left(L^{1}(G),\|\cdot\|_1,  \ast \right)$ is a Banach algebra \cite{Folland, Larsen}.
\begin{definition}  A  continuous function 	$\omega : G \rightarrow (0,\infty)$ is called a weight on $G$. 
\end{definition}

Let $\omega$ be a  weight on $G$. Consider the weighted Lebesgue space $L^{1}_{\omega}(G)$ defined as 

$$ L^{1}_{\omega}(G):= \left\lbrace f: G \rightarrow \mathbb{C}: \int_{G}|f(x)| \omega(x)dx < \infty\right\rbrace.$$
On this space, one defines the norm	 
	  $ \left\| f \right\|_{1,\omega} = \displaystyle\int_{G}\left| f(x)\right| \omega(x)dx$. The following  generalized convolution product  (\cite{Mahmoodi})   is defined on $L^{1}_{\omega}(G)$   by setting
	 
$$(f \ast _{\omega} g)(x) = \int_{G}^{}f(y)g(y^{-1}x)\dfrac{\omega(y)\omega(y^{-1}x)}{\omega(x)}dy.$$
	
Let us observ that if  $\omega\equiv 1$, then the classical convolution product is  recovered.   
	
The space $\mathcal{L}^{1}_{\omega}(G)=\left(L^{1}_{\omega}(G), \left\| \cdot\right\|_{1,\omega},\ast _{\omega}\right)$ is a faithful Banach algebra  called  {\it the weighted group algebra} in \cite{Mahmoodi}. 

Multipliers of $\mathcal{L}^{1}_{\omega}(G)$ when $G$ is a locally compact abelian group had been studied in \cite{Issa 2}.  There, the authors  defined the multipliers of the space $\mathcal{L}^{1}_{\omega}(G)$ via  the introduction of  the operator $\Gamma^{s}_{\omega}, s\in G$ defined by  $$ \Gamma^{s}_{\omega}f(x) = \dfrac{\tau_{s}\mathcal{M}_{\omega}f(x)}{\omega(x)},$$ where 
  $\mathcal{M}_{\omega}$
   is the multiplication operator defined by 
   
   $$(\mathcal{M}_{\omega}f)(x) = \omega(x)f(x)$$
   and $\tau_{s}$ is the translation operator defined by 
  $$ (\tau_{s}f)(x) = f(s^{-1}x), f\in \mathcal{L}^{1}_{\omega}(G).$$
 The definition of a multiplier is stated as follows.
 \begin{definition}\cite{Issa 2}
 	A map $T:  \mathcal{L}_\omega ^1 (G)\rightarrow \mathcal{L}_\omega ^1 (G)$ is called a {\rm multiplier} if $T$ is linear, bounded and commutes with the operator $\Gamma_\omega^s$ for all $s\in G$. 
 \end{definition}
 We denote by $M(\mathcal{L}_\omega ^1 (G))$  the set of all multipliers on $ \mathcal{L}_\omega ^1 (G) $. The following characterization of multipliers was  obtained. 

\begin{theorem}\cite{Issa 2}
 	Let $G$ be a locally compact abelian group. Let $T:  \mathcal{L}_\omega ^1 (G)\rightarrow \mathcal{L}_\omega ^1 (G)$ be a linear map. Then, $T\in M(\mathcal{L}_\omega ^1 (G))$ if and only if $ T(f \ast_{\omega} g)=Tf \ast_{\omega} g=f \ast_{\omega} Tg $ for all $ f, g \in \mathcal{L}^{1}_{\omega}(G)$. 
 \end{theorem}
 
 \subsection{{Gelfand pairs}}
In this subsecton,  we recall the notion of   Gelfand pair on a locally compact group and  state some related properties.  For interested readers, we refer  \cite{Dieudonne, Dijk, Faraut, Kangni2} and references therein. Let $G$ be  a locally compact group and let $K$ be a compact subgroup of $G$. A function $f: G\rightarrow\mathbb{C}$ is said to be left $K$-invariant if 
 $$\forall x\in G,  \forall k \in K, f(kx)=f(x).$$
 
  A  right $K$-invariant function is defined in an obvious way. 
 A function $f: G\rightarrow\mathbb{C}$ is said to be  $K$-bi-invariant if  $f$ is both left and right $K$-invariant. That is, 
 	$$\forall x\in G, \forall k_1, k_2 \in K,  f(k_1xk_2)=f(x).$$
 	
 We denote by $C_c(G)$ the set complex continuous functions on $G$  with compact support and  by  $C_{c}(K\setminus G/K)$ the members of $C_c(G)$ which are $K$-bi-invariant. We also denote by  $L^1(K\setminus G/K)$, the set of complex $K$-bi-invariant integrable functions on $G$. A couple  $(G,K)$ is called a Gelfand pair if $C_{c}(K\setminus G/K)$ is a commutative  algebra under the classical convolution product. By a density argument (in fact, $C_{c}(K\setminus G/K)$ is a dense subset of  $L^1(K\setminus G/K)$), it is equivalent to say that $(G,K)$ is called a Gelfand pair if $\left(L^1(K\setminus G/K),\ast\right)$  is commutative.
  
 It is well known in the litterature  that if $f \in C_{c}(G)$, then the  map   $f\longmapsto f^{\sharp}$ where  $f^\sharp$ is defined by 

\begin{equation}\label{re1}
f^{\sharp}(x)= \displaystyle \int_{K}^{}\int_{K}^{}f(k_{1}xk_{2})dk_{1}dk_{2},
\end{equation}
is a projection from $C_{c}(G)$ into $C_{c}(K \setminus G/K)$ (see for instance \cite{Dijk}).
 Moreover,  the  function $f^{\sharp}$ is $ K$-bi-invariant and if $f$ is integrable, then 
$$\displaystyle\int_G  f^{\sharp}(x)dx=\displaystyle\int_G  f(x)dx.$$
  
\section{Main results}\label{Mainresults}

\subsection{Weighted Gelfand pairs}\label{Weighted Gelfand pairs}
We introduce the  concept of { weighted Gelfand pair} as one of the possible generalization   of the notion of (classical) Gelfand pair. 
Let $G$ be a locally compact group   and let  $K$ be  a compact subgroup of $G$. 
We start this section with the following  properties related to the generalized convolution product $\ast_{\omega}$.	

\begin{theorem}
Let $G$ be a locally compact group and let $K$ be a compact subgroup of $G$. Let $\omega$ be a $K$-bi-invariant weight  on $G$.  Then, $\forall f \in C_{c}(K\setminus G/K),\, \forall h\in C_{c}(G)$,
$$(h\ast _{\omega}f)^{\sharp} = h^{\sharp}\ast _{\omega}f.$$
\end{theorem}
\begin{proof}
In one hand, we have
\begin{align*}
(h\ast _{\omega}f)^{\sharp}(x)&=  \int_{K}^{}\int_{K}(h\ast _{\omega}f)(k_{1}xk_{2})dk_{1}dk_{2} \\
					& = \int_{K}^{}\int_{K} \left[ \int_{G}h(y)f(y^{-1}k_{1}xk_{2})\frac{\omega(y)\omega(y^{-1}k_{1}xk_{2})}{w(k_{1}xk_{2})}\ dy \right]  dk_{1}dk_{2}  \\		 
			&=  \int_{G}^{}\int_{K}  \int_{K}h(y)f(y^{-1}k_{1}x)\frac{\omega(y)\omega(y^{-1}k_{1}x)}{\omega(x)}dy dk_{1}dk_{2}(\mbox{Fubini's theorem}).   
\end{align*}
	In the other hand,
\begin{align*}
			(h^{\sharp}\ast _{\omega}f)(x)&=  \int_{G}h^{\sharp}(y)f(y^{-1}x)\frac{\omega(y)\omega(y^{-1}x)}{\omega(x)}dy \\
			& = \int_{G} \left[ \int_{K} \int_{K}h(k_{1}yk_{2})dk_{1}dk_{2} \right] f(y^{-1}x) \frac{\omega(y)\omega(y^{-1}x)}{\omega(x)}dy. 
			\end{align*}
			
 Making the change of variable $y\longrightarrow k_{1}^{-1}yk_{2}^{-1}$, we have
 \begin{align*}
(h^{\sharp}\ast _{\omega}f)(x)	&=  \int_{G}  \int_{K} \int_{K}h(y)  f(k_{2}y^{-1}k_{1}x) \frac{\omega(k_{1}^{-1}yk_{2}^{-1})\omega(k_{2}y^{-1}k_{1}x)}{\omega(x)}dydk_{1}dk_{2} \\ 	
		&= \int_{G}  \int_{K} \int_{K}h(y)  f(y^{-1}k_{1}x)\frac{\omega(y)\omega(y^{-1}k_{1}x)}{\omega(x)}dydk_{1}dk_{2}.
	\end{align*}
Thus, 	$(h\ast _{\omega}f)^{\sharp} = h^{\sharp}\ast _{\omega}f$.
\end{proof}
\begin{theorem}
Let $G$ be a locally compact group and let $K$ be a compact subgroup of $G$. Let $\omega$ be a $K$-bi-invariant weight  on $G$. Let $f \in C_{c}(K\setminus G/K)$ and $h\in C_{c}(G)$. If $h$ is  left  $K$-invariant, then 
$$
(h\ast _{\omega}f)^{\sharp} = h\ast _{\omega}f^{\sharp}.
$$
\end{theorem}
\begin{proof}
\begin{align*}
			(h\ast _{\omega}f^{\sharp})(x)&=    \int_{G}h(y)f^{\sharp}(y^{-1}x)\frac{\omega(y)\omega(y^{-1}x)}{\omega(x)}dy \\
			& = \int_{G}h(y) \left[ \int_{K} \int_{K}f(k_{1}y^{-1}xk_{2})dk_{1}dk_{2} \right]  \frac{\omega(y)\omega(y^{-1}x)}{\omega(x)}dy  \\
			&=    \int_{G}  \int_{K} \int_{K}h(y)  f(k_{1}y^{-1}xk_{2}) \frac{\omega(y)\omega(y^{-1}x)}{\omega(x)}dydk_{1}dk_{2}  
			\end{align*} 
Making the change of variable     $y\longrightarrow  k_{1}^{-1}y$, we have
\begin{align*}
(h\ast _{\omega}f^{\sharp})(x)&=    \int_{G}  \int_{K} \int_{K}h(k^{-1}_{1}y)  f(y^{-1}k_{1}x) \frac{\omega(k^{-1}_{1}y)\omega(y^{-1}k_{1}x)}{\omega(x)}dydk_{1}dk_{2}  \\ 	
			&=   \int_{G}  \int_{K} \int_{K}h(k^{-1}_{1}y)  f(y^{-1}k_{1}x) \frac{\omega(y)\omega(y^{-1}k_{1}x)}{\omega(x)}dydk_{1}dk_{2}\\  
	&=   \int_{G}  \int_{K} \int_{K}h(y)  f(y^{-1}x) \frac{\omega(y)\omega(y^{-1}x)}{\omega(x)}dydk_{1}dk_{2}\\
		&=   \int_{G}  \int_{K} \int_{K}h(y)  f(k_1y^{-1}xk_2) \frac{\omega(y)\omega(k_1y^{-1}xk_2)}{\omega(x)}dydk_{1}dk_{2}\\	
			&=  \int_{K} \int_{K}h\ast _{\omega}f(k_1y^{-1}xk_2)dk_{1}dk_{2}\\			
	&=(h\ast _{\omega}f)^{\sharp}(x).
	\end{align*} 
	Thus, $ (h\ast _{\omega}f)^{\sharp} = h\ast _{\omega}f^{\sharp}$.
	\end{proof}
Let $\theta$ be  a continuous involutive automorphism of $G$. For a function $f$ defined on $G$, set  $f^\theta (x)=f(\theta (x))$.
\begin{theorem}\label{theta}
Let $G$ be a locally compact group. Let $\omega$ be a weight  on $G$.
If $\omega^{\theta} =\omega$, then $\forall f,h \in C_{c}(G)$, $$f^{\theta}\ast _{\omega}h^{\theta} = (f \ast _{\omega}h)^{\theta}.$$
\end{theorem}
\begin{proof}	
\begin{align*}
	 f^{\theta}\ast_{\omega}h^{\theta}(x) & = \displaystyle\int_{G}f^{\theta}(y)h^{\theta}(y^{-1}x)\dfrac{\omega(y)\omega(y^{-1}x)}{\omega(x)}dy \\
	 & =  \int_{G}f^{\theta}(y)h^{\theta}(y^{-1}x)\dfrac{\omega^{\theta}(y)\omega^{\theta}(y^{-1}x)}{\omega^{\theta}(x)}dy \\
	 &= \int_{G}^{}f(\theta(y))h(\theta(y^{-1}x))\dfrac{\omega(\theta(y))\omega(\theta(y^{-1}x)}{\omega(\theta(x))}dy \\
	 &= \int_{G}^{}f(\theta(y))h(\theta(y^{-1})\theta(x))\dfrac{\omega(\theta(y))\omega(\theta(y^{-1})\theta(x))}{\omega(\theta(x))}dy \\
	 &= \int_{G}^{}f(\theta(y))h([\theta(y)]^{-1}\theta(x))\dfrac{\omega(\theta(y))\omega([\theta(y)]^{-1}\theta(x))}{\omega(\theta(x))}dy.
\end{align*}
Making the change of variable  $s=\theta(y)$, we have
\begin{align*}
f^{\theta}\ast_{\omega}h^{\theta}(x) & = \int_{G}^{}f(s)h(s^{-1}\theta(x))\dfrac{\omega(s)\omega(s^{-1}\theta(x))}{\omega(\theta(x))}ds\\
	 &= \left( f \ast_{\omega}h\right) \theta(x)\\
	 &= \left( f \ast_{\omega}h\right)^{\theta} (x).
\end{align*}	
Thus, $f^{\theta}\ast _{\omega}h^{\theta} = (f \ast _{\omega}h)^{\theta}$.
\end{proof}
\begin{remark}{ \rm
If for instance we choose $\theta$ as the inverse map of $G$, that is $\theta (x)=x^{-1}$, then the condition $\omega^\theta=\omega$ is satisfied by symmetric weights. 
}\end{remark}
 
Set $\mathcal{L}^{1}_{\omega}(K\setminus G/K):=\{f\in \mathcal{L}^{1}_{\omega}(G): f \, \text{is K-bi-invariant} \}.$
When equipped with the norm $\lVert \cdot \lVert_{1,\omega}$ and  the generalized convolution $\ast_{\omega}$, the space $\mathcal{L}^{1}_{\omega}(K\setminus G/K)$ is  Banach algebra (non necessarily commutative). 
\begin{theorem}\label{dense}
Let $G$ be a locally compact group and let $K$ be a compact subgroup of $G$. Let $\omega$ be a $K$-bi-invariant weight  on $G$. Then, 
	the space $C_{c}(K\setminus G/K)$ is dense in $\mathcal{L}^{1}_{\omega}(K\setminus G/K)$.
\end{theorem}
\begin{proof}
Let $f$ be in $\mathcal{L}^{1}_{\omega}(K\setminus G/K)$. Then,  $f\omega \in L^1(K\setminus G/K)$. Since 	$C_{c}(K\setminus G/K)$ is dense in $L^1(K\setminus G/K)$, then there exists a sequence $(g_n)$  of elements of $C_{c}(K\setminus G/K)$ such that $\lim\limits_{n \to +\infty}{\|g_n-f\omega \|}_{{L^{1}(K\setminus G/K)}}=0.$ Let us  set $f_n=\dfrac{g_n}{\omega}$. Then,  $(f_n)_{n\in\mathbb{N}} \subset C_{c}(K\setminus G/K)$ and 
$$
\lim_{n \to +\infty}{\lVert f_n - f \rVert}_{\mathcal{L}^1_{\omega}(K\setminus G/K)}= \lim_{n \to +\infty}{\lVert g_n - f\omega \rVert_{L^{1}(K\setminus G/K)}}
=0.
$$
Thus,  $C_{c}(K\setminus G/K)$ is dense in $\mathcal{L}^{1}_{\omega}(K\setminus G/K)$.
\end{proof}
The following definition is one of the main ingredients of the article. It is concerned by a generalization of the notion of  Gelfand pair which we call a weighted Gelfand pair. 
\begin{definition}
Let $G$ be a locally compact group. Let $K$ be a compact subgroup of $G$. Let $\omega$ be a $K$-bi-invariant weight  on $G$.  The triple $(G,K,\omega)$ is called   a {\rm weighted Gelfand pair}   if the space $\left(\mathcal{L}^{1}_{\omega}(K \setminus G/ K), \ast_{\omega}\right)$ is  commutative.  
\end{definition}

\begin{remark}{\rm
Let us notice that if $\omega\equiv 1$, then one  recovers the classical notion of Gelfand pair.}
\end{remark}

For a complex  valued function  $f$ defined on $G$, set 
$\check{f} (x)=f(x^{-1})$. We state the following result. 
\begin{theorem}\label{aut}
Let $G$ be a locally compact group and let $K$ be a compact subgroup of $G$. If there exists an involutive continuous  automorphism $\theta$ of $G$ such that for all $x \in G$, $\theta(x) \in Kx^{-1}K$,  then for each  integrable $K$-bi-invariant function $f$ on $G$, we have $ f^{\theta}= \check{f}.$
	\end{theorem}
\begin{proof}
	Let $x\in G$. By definition,    $f^{\theta}(x) = f(\theta(x))$. If  $\theta(x) \in Kx^{-1}K$, then  there exists   $ k_{1}, k_{2}$ in   $K$  such that   $\theta(x) = k_{1}x^{-1}k_{2}$. Therefore, 
	\begin{eqnarray*}
		f^{\theta}(x)  =  f(\theta(x))
		 =  f(k_{1}x^{-1}k_{2})
		=  f(x^{-1})=\check{f}(x).
	\end{eqnarray*}
	Thus, $f^{\theta} = \check{f}.$
\end{proof}

The following result state a necessary condition to have a weighted Gelfand pair.
\begin{theorem}\label{rap} Let $G$ be a locally compact group  and let be $K$ a compact subgroup of $G$. Let 
$\omega$ be a  $K$-bi-invariant weight such that $\omega^{\theta}=\omega$. If there exists a continuous involutive automorphism $\theta$ of $ G$ such that $ \theta(x) \in Kx^{-1}K$ for all $x\in G$, then 
$(G,K, \omega) $ is a weighted Gelfand pair.
\end{theorem}

\begin{proof}
By Theorem \ref{theta}, one has $\forall f, h \in \mathcal{L}_{\omega}^{1}(K \setminus G/K), f^{\theta}\ast_{\omega}h^{\theta}  = \left( f \ast_{\omega}h\right)^{\theta}$. Since
$\omega^{\theta}= \omega$    and  $\theta^{2}$ is the identity map of $G$, then $\forall x\in G$,  $$\omega(x) =   \omega(\theta(x))
 =  \omega(\theta^{-1}(x))
=  \omega(\theta(x^{-1}))
 =  \omega(x^{-1}).$$
That is to say $\omega$ is symmetric. Furthermore, 
\begin{align*}
(\check{f}\ast_{\omega}\check{h})\check{}(x) & =  (\check{f}\ast_{\omega}\check{h})(x^{-1})\\
& =  \int_{G}\check{f}(y)\check{h}(y^{-1}x^{-1})\dfrac{\omega(y)\omega(y^{-1}x^{-1})}{\omega(x^{-1})}dy\\
& = \int_{G}\check{f}(y)\check{h}\left[ (xy)^{-1}\right] \dfrac{\omega(y)\omega\left[ (xy)^{-1}\right]}{\omega(x^{-1})}dy\\
& =  \int_{G}\check{f}(y)h (xy) \dfrac{\omega(y)\omega (xy)}{\omega(x)}dy\\
& =  \int_{G}f(y^{-1})h (xy) \dfrac{\omega(y)\omega (xy)}{\omega(x)}dy\\
& =  (h\ast_{\omega}f)(x).
\end{align*}

Then, $ \left( \check{f}\ast_{\omega}\check{h}\right)\check{{}} = h\ast_{\omega}f$.  Therefore, $\check{f}\ast_{\omega}\check{h} = (h\ast_{\omega}f)\check{{}}$. By Theorem \ref{aut}, we have
$$f^{\theta}\ast_{\omega}h^{\theta} = (h\ast_{\omega}f)^{\theta}.$$
However, $\check{f}\ast_{\omega}\check{h} = f^{\theta}\ast_{\omega}h^{\theta} = (f\ast_{\omega}h)^{\theta}$ by Theorem \ref{theta}.
It follows that  $(h \ast_{\omega}f)^{\theta}=(f \ast_{\omega} h)^{\theta}$. Therefore,  $h \ast_{\omega}f=f \ast_{\omega} h$. Thus,
$(G, K, \omega)$ is a weighted Gelfand pair.
\end{proof}

Set $\check{\omega}(x)= \omega(x^{-1})$. Then,  
$\mathcal{M}_{\omega\check{\omega}}f(x)=f(x)\omega(x^{-1})\omega(x),   x\in G$.
\begin{theorem} Let $G$ be a locally compact group and let $K$ be a compact subgroup of $G$. Let $\omega$ be a $K$-bi-invariant weight on $G$ with $\omega(e)=1$. 
		If $(G,K,\omega)$ is a weighted Gelfand pair, then $\forall f\in C_{c}(K\setminus G/K)$,
 \begin{equation}\label{implies unimodularity}		 \int_{G}\mathcal{M}_{\omega\check{\omega}}f(x)dx= \int_{G}\mathcal{M}_{\omega\check{\omega}}f(x^{-1})dx.
\end{equation}	
\end{theorem}
\begin{proof}
Let $g \in C_{c}(K\setminus G/K)$ be such that $g\equiv 1$ over the compact set $Supp(f)\cup~(Supp(f))^{-1}$. Then,
		\begin{align*}
		\int_{G}\mathcal{M}_{\omega\check{\omega}}f(x)dx &=\int_{G}f(x)\omega(x)\omega(x^{-1})dx\\
	 & = \int_{G} f(x)g(x^{-1}e)\dfrac{\omega(x)\omega(x^{-1}e)}{\omega(e)}dx \\
	                           & =   (f\ast_{\omega}g)(e) \\
	                           &=    (g\ast_{\omega} f)(e)\\
	                           & =  \int_{G} f(x^{-1}e)g(x)\dfrac{\omega(x^{-1}e)\omega(x)}{w(e)}dx\\
	                           & =  \int_{G} f(x^{-1})\omega(x^{-1})\omega(x)dx \\
	                           &= 
	  \int_{G}\mathcal{M}_{\omega\check{\omega}}f(x^{-1})dx.
\end{align*}
\end{proof}
\begin{remark}{\rm
If $\omega\equiv 1$, then the equality (\ref{implies unimodularity}) implies $\displaystyle\int_{G}f(x)dx= \displaystyle\int_{G}f(x^{-1})dx$. Therefore, we recover the classical  fact which states that  if  $(G,K)$ is a Gelfand pair, then $G$ is a unimodular group.} 
\end{remark}

\subsection{Weighted spherical functions and weighted  spherical Fourier transform}\label{Weighted spherical functions and weighted Fourier transform}
Let $(G,K)$ be a Gelfand pair. A spherical function for $(G,K)$   is a $K$-bi-invariant continuous function $\varphi$ on $G$, such that the map $\chi : f \longmapsto \chi(f)= \displaystyle \int_{G} f(x)\varphi(x^{-1})dx$ is a non-zero continuous character of the convolution algebra $C_{c}(K\setminus G/K)$; that is,  $\chi$ is linear, continuous and $\chi(f\ast g) = \chi(f)\chi(g), \forall f,g \in C_{c}(K\setminus G/K)$. 
It is equivalent to say that a non-zero continuous $K$-bi-invariant function $\varphi$ is a spherical function for the Gelfand pair $(G,K)$ if and only if $\displaystyle\int_K\varphi (xky)dk=\varphi (x)\varphi (y)$, (see for instance \cite{Dijk, Kangni2}). 
In what follows, we state the suitable definition of a spherical function  for the  convolution algebra $\mathcal{L}^{1}_{\omega}(K \setminus G/ K)$. Moreover, the appropriate Fourier transform in this new context  is considered.

\begin{definition}
Let $(G,K,\omega)$ be a weighted Gelfand pair. Let $\varphi$ be a non-zero $K$-bi-invariant. Set $\chi_{\varphi,\omega}(f)=\displaystyle\int_{G}\mathcal{M}_{\omega}f(x)\mathcal{M}_{\omega}\varphi(x^{-1})dx$. 
The function $\varphi$  is called an $\omega$-spherical function on $G$ if $\chi_{\varphi,\omega}$  verifies the property 
$\chi_{\varphi,\omega} (f\ast_\omega g)= \chi_{\varphi,\omega}(f)\chi_{\varphi,\omega}(g),\, \forall f,g \in C_{c}(K\setminus G/K)$. 
\end{definition}
 
\begin{theorem}\label{sphe}
	Let $(G,K,\omega)$ be a weighted Gelfand pair.
	Let $\varphi$ be a non-zero $K$-bi-invariant continuous function on $G$.  The function $\varphi : G\longrightarrow\mathbb{C}$ is    an $\omega$-spherical function if and only if $\forall x,y\in G$,
	 $$\int_{K} \mathcal{M}_{\omega}\varphi(xky)dk = \mathcal{M}_{\omega}\varphi(x) \mathcal{M}_{\omega}\varphi(y).$$
\end{theorem}

\begin{proof}
 Let  $f,g$ be in $C_{c}(K\setminus G/K)$.
\begin{align*}
\chi_{\varphi,\omega}(f\ast_{\omega}g) &= \int_{G} \mathcal{M}_{\omega}(f\ast_{\omega}g)(x)\mathcal{M}_{\omega}\varphi(x^{-1})dx \\
 &=\int_{G} (f\ast_{\omega}g)(x)\mathcal{M}_{\omega}\varphi(x^{-1})\omega(x)dx\\
	 				 &=  \int_{G} \int_{G} f(y)g(y^{-1}x)\mathcal{M}_{\omega}\varphi(x^{-1})\omega(y)\omega(y^{-1}x)dxdy. 
\end{align*}
By the change of variables  $x \longrightarrow yx$,  we have
\begin{align*}
	 \chi_{\varphi,\omega}(f\ast_{\omega}g)&= \int_{G} \int_{G} \mathcal{M}_{\omega}f(y)\mathcal{M}_{\omega}g(x)\mathcal{M}_{\omega}\varphi(x^{-1}y^{-1})dxdy \\		
 &= \int_{K}dk\int_{G} \int_{G} \mathcal{M}_{\omega}f(y)\mathcal{M}_{\omega}g(x)\mathcal{M}_{\omega}\varphi(x^{-1}ky^{-1})dxdy. \\
& = 	\int_{G} \int_{G}\mathcal{M}_{\omega} f(y)\mathcal{M}_{\omega}g(x)\left( \int_{K}\mathcal{M}_{\omega}\varphi(x^{-1}ky^{-1})dk\right)dxdy \\
	 & =	\int_{G} \int_{G}\mathcal{M}_{\omega} f(y^{-1})\mathcal{M}_{\omega}g(x^{-1})\left( \int_{K}\mathcal{M}_{\omega}\varphi(xky)dk\right)dxdy.	 
	\end{align*}
	Moreover,
\begin{align*}
 \chi_{\varphi,\omega}(f)\chi_{\varphi,\omega}(g)	
	& = \left( \int_{G}\mathcal{M}_{\omega} f(y)\mathcal{M}_{\omega}\varphi(y^{-1})dy\right) \left(  \int_{G}\mathcal{M}_{\omega}g(x)\mathcal{M}_{\omega}\varphi(x^{-1})dx \right) \\
	& =	\int_{G} \int_{G}\mathcal{M}_{\omega} f(y)\mathcal{M}_{\omega}g(x)\left(\mathcal{M}_{\omega}\varphi(x^{-1}) \mathcal{M}_{\omega}\varphi(y^{-1})\right)dxdy \\
	& = 	\int_{G} \int_{G}\mathcal{M}_{\omega} f(y^{-1})\mathcal{M}_{\omega}g(x^{-1})\left(\mathcal{M}_{\omega}\varphi(x) \mathcal{M}_{\omega}\varphi(y)\right)dxdy.
\end{align*}
\begin{flushleft}
	$\chi_{\varphi,\omega}(f\ast_{\omega}g)-\chi_{\varphi,\omega}(f)\chi_{\varphi,\omega}(g)=$
\end{flushleft}
\begin{eqnarray*}
	\int_{G} \int_{G}\mathcal{M}_{\omega} f(y^{-1})\mathcal{M}_{\omega}g(x^{-1})\left( \int_{K}\mathcal{M}_{\omega}\varphi(xky)dk - 
	\mathcal{M}_{\omega}\varphi(x)
	\mathcal{M}_{\omega}\varphi(y)\right)dxdy .
\end{eqnarray*}
The function  $\varphi$ is an $\omega$-spherical function if and only if
 $\forall f,g \in C_{c}(K\setminus G/K),\,\chi_{\varphi,\omega}(f\ast_{\omega}g)= \chi_{\varphi,\omega}(f)\chi_{\varphi,\omega}(g)$. That is,
 $\displaystyle\int_{K} \mathcal{M}_ {\omega}\varphi(xky)dk = \mathcal{M}_{\omega}\varphi(x) \mathcal{M}_{\omega}\varphi(y)$. 
\end{proof} 

\begin{theorem}
	Let $(G,K,\omega)$ be a weighted Gelfand pair. Let $\varphi$ be a continuous $K$-bi-invariant function on $G$. Then, $\varphi$ is an $\omega$-spherical function if and only if the following two conditions hold.
	\begin{enumerate}
		\item \label{numero 1} $\varphi(e)=1$,
		\item \label{numero 2} For each $ f\in C_{c}(K\setminus G/K)$, there exists a scalar $\chi_{\varphi,\omega}(f)$ such that  $f \ast_{\omega} \varphi =  \chi_{\varphi,\omega}(f)\varphi$.
	\end{enumerate}
\end{theorem}
\begin{proof}
	Let us assume that $\varphi$ is a $\omega$-spherical function. Then, 
	$\displaystyle\int_{K}\mathcal{M}_{\omega}\varphi(xke)dk =\mathcal{M}_{\omega}\varphi(x)\mathcal{M}_{\omega}\varphi(e)$. We  used the $K$-right-invariance of the functions $\varphi$ and $\omega$, the condition $\omega (e)=1$   and  the fact that the Haar measure of $K$ is normalized, to obtain
$ \mathcal{M}_{\omega}\varphi(x)= \mathcal{M}_{\omega}\varphi(x)\times \varphi(e)$. Then, 
$\forall x\in G,   \mathcal{M}_{\omega} \varphi(x) \left[ \varphi(e)-1\right] =0$. This implies  $\varphi(e)=1$.
In the other hand, 
	
\begin{align*}
			f \ast_{\omega} \varphi(x)&= \int_{G}f(y)\varphi(y^{-1}x)\dfrac{\omega(y)\omega(y^{-1}x)}{\omega(x)}dy \\
			&= \int_{K}dk \int_{G}f(y)\varphi(y^{-1}x)\dfrac{\omega(y)\omega(y^{-1}x)}{\omega(x)}dy \\ 
			&= \int_{K}dk \int_{G}f(y)\omega(y)\mathcal{M}_{\omega}\varphi(y^{-1}x)\dfrac{dy}{\omega(x)} \\
			&=  \int_{G}f(y)\omega(y)\left( \int_{K}\mathcal{M}_{\omega}\varphi(y^{-1}kx)dk\right) \dfrac{dy}{\omega(x)} \\ 
			&= \int_{G}f(y)\omega(y) \mathcal{M}_{\omega}\varphi(y^{-1})\mathcal{M}_{\omega}\varphi(x) \dfrac{dy}{\omega(x)} \\
			&= \dfrac{\mathcal{M}_{\omega}\varphi(x)}{\omega(x)}\int_{G}f(y)\mathcal{M}_{\omega}\varphi(y^{-1})\omega(y)dy\\
				&= \varphi(x)\chi_{\varphi,\omega}(f).
		\end{align*}
	Thus, $f \ast_{\omega} \varphi(x)=\varphi(x)\chi_{\varphi,\omega}(f).$
	
 In the converse, let us assume \ref{numero 1} and \ref{numero 2}. We may show that  $\chi_{\varphi,\omega}(f\ast_{\omega}g) = \chi_{\varphi,\omega}(f)\chi_{\varphi,\omega}(g)$. 

\begin{align*}
\chi_{\varphi,\omega}  (f\ast_{\omega}g) & =  \int_{G} f\ast_{\omega}g(x)\mathcal{M}_{\omega}\varphi(x^{-1})\omega(x)dx \\
		 &=  \int_{G} \int_{G} f(y)g(y^{-1}x)\mathcal{M}_{\omega}\varphi(x^{-1})\omega(y)\omega(y^{-1}x)dxdy\\
		 & =  \int_{G} \left( f(y)\omega(y)\right) \int_{G} g(y^{-1}x)\varphi(x^{-1})\frac{\omega(y^{-1}x)\varphi(x^{-1})}{\omega(y^{-1})}\omega(y^{-1})dxdy\\
		 &= \int_{G}f(y)\omega(y)\omega(y^{-1}) \left(g\ast_{\omega}\varphi \right)(y^{-1})dy\\
		 &= \int_{G}f(y)\omega(y)\omega(y^{-1})\varphi(y^{-1})\chi(g)dy\\
		 &= \left( \int_{G}\mathcal{M}_{\omega}f(y) \mathcal{M}_{\omega}\varphi(y^{-1})dy\right) \chi(g)\\
		 &= \chi_{\varphi,\omega}(f)\chi_{\varphi,\omega}(g).
\end{align*}
\end{proof}

We denote by $S_{\omega}(G,K)$  the set of $\omega$-spherical functions for the weighted Gelfand pair $(G, K,\omega)$. We set  
$$S^{b}_{\omega}(G,K) = \left\lbrace \varphi \in S_{\omega}(G,K) \,:\,  \mathcal{M}_{\omega}\varphi  \mbox{ is  bounded} \right\rbrace.$$

\begin{theorem}
	Each non-zero character $\chi$ of $\mathcal{L}^{1}_{\omega
	}(K\setminus G/K)$ is of the form 
	\begin{eqnarray}
	\chi(f)= \int_{G}\mathcal{M}_{\omega}f(x)\mathcal{M}_{\omega}\varphi(x^{-1})dx
	\end{eqnarray}
	where $\varphi \in S^{b}_{\omega}(G,K)$.
	\end{theorem}

\begin{proof}
	Let $\chi$ be a non-zero character of $\mathcal{L}^{1}_{\omega
	}(K\setminus G/K)$. Then, $\chi$ is a continuous linear form on $\mathcal{L}^{1}_{\omega
	}(K\setminus G/K)$. Then, there exist a function $\varphi \in  L^{\infty}_{\frac{1}{{\omega}}}(K\setminus G/K)$ such that 
$$\chi(f)= \int_{G}\mathcal{M}_{\omega}f(x)\mathcal{M}_{\omega}\varphi(x^{-1})dx.$$
Let us show that $\varphi$ is a weighted spherical function.
For all  $f, g \in C_{c}(K\setminus G/K)$. We have
	\begin{align*}
		0&=\chi(f\ast_{\omega}g)-\chi(f)\chi(g)\\
		&=\int_{G} \int_{G}\left[\mathcal{M}_{\omega}\varphi(xy)-\mathcal{M}_{\omega}\varphi(x)\mathcal{M}_{\omega}\varphi(y)\right]\mathcal{M}_{\omega}f(x)\mathcal{M}_{\omega}g(y)dxdy \\		
&= \int_{G} \int_{G}\left[\int_K\mathcal{M}_{\omega}\varphi(xky)dk-\mathcal{M}_{\omega}\varphi(x)\mathcal{M}_{\omega}\varphi(y) \right]\mathcal{M}_{\omega}f(x^{-1})\mathcal{M}_{\omega}g(y^{-1})dxdy.	
	\end{align*}
Therefore, $\displaystyle\int_K\mathcal{M}_{\omega}\varphi(xky)dk=\mathcal{M}_{\omega}\varphi(x)\mathcal{M}_{\omega}\varphi(y) $. Thus, $\varphi$ is a weighted spherical function.
\end{proof}
We are now able to introduce the notion of weighted spherical Fourier transform. 
\begin{definition} Let $(G,K,\omega)$ be a weighted Gelfand pair. 
	The weighted spherical Fourier transform  of a function $f \in \mathcal{L}_{\omega}(K\setminus G/K)$ is denoted by $$ \mathcal{F}^{\omega}(f)(\varphi)= \int_{G}\mathcal{M}_{\omega}f(x)\mathcal{M}_{\omega}\varphi(x^{-1})dx,\, \varphi \in  S^b_{\omega}(G,K).$$
\end{definition}

The weighted spherical Fourier transform satisfies the following convolution theorem.
\begin{theorem}
 \begin{equation}\label{Fourier} 
\forall f,g \in \mathcal{L}_{\omega}(K\setminus G/K),\,\mathcal{F}^{\omega}(f\ast_{\omega}g) = \mathcal{F}^{\omega}(f)\mathcal{F}^{\omega}(g).\end{equation}	
\end{theorem}
\begin{proof}
Let $\varphi \in S^b_{\omega}(G,K).$
	\begin{align*}
\mathcal{F}^{\omega}(f\ast_{\omega}g)(\varphi) & =  \int_{G}\mathcal{M}_{\omega}(f\ast_{\omega}g)(x)\mathcal{M}_{\omega}\varphi(x^{-1})dx\\
& =  \int_{G}\int_{G}f(y)g(y^{-1}x)\mathcal{M}_{\omega}\varphi(x^{-1})w(y)w(y^{-1}x)dxdy.
	\end{align*}
	By the change of variable  $s \longrightarrow y^{-1}x$,  we have
\begin{align*}
\mathcal{F}^{\omega}(f\ast_{\omega}g)(\varphi) & =  \int_{G}\int_{G}f(y)g(s)\mathcal{M}_{\omega}\varphi(s^{-1}y^{-1})w(y)w(s)dsdy\\
&= \int_{G}\int_{G}f(y)g(s)\left(\int_{K} \mathcal{M}_{\omega}\varphi(s^{-1}ky^{-1})dk\right) w(y)w(s)dsdy\\
&=  \int_{G}\int_{G}f(y)g(s)\left( \mathcal{M}_{\omega}\varphi(s^{-1})\mathcal{M}_{\omega}\varphi(y^{-1})\right) w(y)w(s)dsdy\\
& =  \left(  \int_{G}f(y)w(y)\mathcal{M}_{\omega}\varphi(y^{-1})dy\right) \left(\int_{G}g(s)w(s) \mathcal{M}_{\omega}\varphi(s^{-1})ds\right) \\
& =   \left(  \int_{G}\mathcal{M}_{\omega}f(y)\mathcal{M}_{\omega}\varphi(y^{-1})dy\right) \left(\int_{G}\mathcal{M}_{\omega}g(s) \mathcal{M}_{\omega}\varphi(s^{-1})ds\right) \\
& =  \mathcal{F}^{\omega}(f)(\varphi)\mathcal{F}^{\omega}(g)(\varphi).
\end{align*}	
Thus, $\mathcal{F}^{\omega}(f\ast_{\omega}g)  =   \mathcal{F}^{\omega}(f)\mathcal{F}^{\omega}(g)$.
\end{proof}

\begin{theorem}\label{injective}
	The weighted spherical Fourier transform is injective on $\mathcal{L}_{\omega}^{1}(K\setminus G/K)$. This is equivalent to say that the Banach algebra $\mathcal{L}_{\omega}^{1}(K\setminus G/K)$ is semisimple.
\end{theorem}

\begin{proof}
	The spherical Fourier  transform of a function $f\in L^{1}(K\setminus G/K)$ is defined by $$
	\mathcal{F}(f)(\Phi)=\displaystyle\int_Gf(x)\Phi(x^{-1})dx,\,  \Phi \in S^b(G,K),$$
	where $S^b(G,K)$ denotes the set of bounded spherical functions for the Gelfand pair $(G,K)$.
	Let $f \in \mathcal{L}^{1}_{\omega}(K\setminus G/K)$ be such that $\mathcal{F}^{\omega}(f)(\varphi)= 0,\,  \forall \varphi \in S^{b}_{\omega}(G,K)$.  It is clear that $$f \in \mathcal{L}^{1}_{\omega}(K\setminus G/K) \Longleftrightarrow \omega f\in L^1(K\setminus G/K)    $$
	$$\varphi \in S^{b}_{\omega}(G,K) \Longleftrightarrow \omega \varphi\in S(G,K).$$ 
	Then, 
	 \begin{align*}
	\mathcal{F}^{\omega}(f)(\varphi)= 0 &\Longrightarrow \mathcal{F}(\omega f)(\omega\varphi)=0  \\
	&\Longrightarrow \omega f=0, \mbox{ since } L^1(K\setminus G/K) \mbox{ is semisimple} \\
	&\Longrightarrow f=0.	
	\end{align*} 
	Thus, $\mathcal{F^{\omega}}$ is injective.
\end{proof}

\subsection{Multipliers for weighted  Gelfand pairs}\label{Multipliers for weighted  Gelfand pairs}
Our main goal in this section is to characterize the multipliers for the weighted Gelfand pairs $(G,K,\omega)$.

\begin{definition}
	A multiplier for the weighted Gelfand pair $(G,K,\omega)$ is a linear map $ T: \mathcal{L}^{1}_{\omega}(K \setminus G/ K) \rightarrow \mathcal{L}^{1}_{\omega}(K \setminus G/ K)$ such that  
	$$\forall f, g \in \mathcal{L}^{1}_{\omega}(K \setminus G/ K),\,T(f \ast_{\omega} g) = Tf\ast_{\omega} g.$$
	\end{definition}

Let us denote by $M( \mathcal{L}^{1}_{\omega}(K \setminus G/K))$ the set of  multipliers of the weighted Gelfand pair $(G,K,\omega)$. 

\begin{theorem}
	Let $(G,K,\omega)$ be a weighted Gelfand pair. 
	
	Let $ T: \mathcal{L}^{1}_{\omega}(K \setminus G/ K) \rightarrow \mathcal{L}^{1}_{\omega}(K \setminus G/ K)$ be a a linear map. Then, the following  assertions are equivalent.
	\begin{enumerate}
		\item\label{ab} $T$ is a multiplier for $(G,K,\omega)$.
		\item\label{ac}  There exists a unique function $\mathfrak{b}$ defined on $S^{b}_{\omega}(G,K)$ such that 
	$$\forall f \in \mathcal{L}^{1}_{\omega}(K \setminus G/ K),\, \mathcal{F}^{\omega}(Tf) = \mathfrak{b}\mathcal{F}^{\omega}(f).$$
\end{enumerate}
\end{theorem}

\begin{proof}
 We mimic the proof from the Proposition 4.2 of \cite{Issa 1}.
 
 \begin{enumerate}
\item \ref{ab} $\Longrightarrow$ \ref{ac} Assume that $T(f\ast_{\omega}g) = Tf \ast_{\omega}g$ for all $f, g \in \mathcal{L}^{1}_{\omega}(K \setminus G/ K).$  Since $(G, K,\omega)$ is a weighted Gelfand pair, then $\mathcal{L}^{1}_{\omega}(K \setminus G/ K)$ is commutative under the convolution product $\ast_{\omega}$. Therefore, we have $$ Tf\ast_{\omega}g = T(f\ast_{\omega}g) = T(g \ast_{\omega}f ) = Tg \ast_{\omega}f.$$ 
	Using the property of the spherical Fourier transform (\ref{Fourier}) we obtain $$ \mathcal{F}^{\omega}(Tf) \times \mathcal{F}^{\omega}(g) = \mathcal{F}^{\omega}(Tg) \times \mathcal{F}^{\omega}(f).$$
	Let us choose $g$ in $\mathcal{L}^{1}_{\omega}(K \setminus G/ K)$ such that for all weighted spherical functions $\varphi$,  $\mathcal{F}^{\omega}(g)(\varphi) \neq 0$. Now, we define $\mathfrak{b}$ by $ \mathfrak{b} (\varphi) = \dfrac{\mathcal{F}^{\omega}(Tg)(\varphi)}{\mathcal{F}^{\omega}(g)(\varphi) }$. Therefore, we have $\mathcal{F}^{\omega}(Tf)(\varphi) = \mathfrak{b}(\varphi)\mathcal{F}^{\omega}(f)(\varphi)$. Thus, $\mathcal{F}^{\omega}(Tf) = \mathfrak{b}\mathcal{F}^{\omega}(f).$ 
	
	Let us  show the unicity of $\mathfrak{b}$. If $\mathfrak{r}$ is a second function on $S_{\omega}^b(G,K)$ such that $ \mathcal{F}^{\omega}(Tf) = \mathfrak{r}\mathcal{F}^{\omega}(f)= \mathfrak{b}\mathcal{F}^{\omega}(f)$ for all $f\in \mathcal{L}^{1}_{\omega}(K \setminus G/ K)$. Then, the equation $ (\mathfrak{b} - \mathfrak{r})\mathcal{F}^{\omega}(f) =0$ for all $f\in \mathcal{L}^{1}_{\omega}(K \setminus G/ K)$ reveals that $ \mathfrak{b} = \mathfrak{r}$.
	
\item \ref{ac} $\Longrightarrow$ \ref{ab} Let us assume that there exists a function $\mathfrak{b}$ defined on $S_{\omega}^b(G,K)$ such that $\mathcal{F}^{\omega}(Tf) = \mathfrak{b}\mathcal{F}^{\omega}(f),\,\forall f\in \mathcal{L}^{1}_{\omega}(K \setminus G/ K)$. For $f,g \in \mathcal{L}^{1}_{\omega}(K \setminus G/ K)$, we have $f\ast_{\omega} g\in\mathcal{L}^{1}_{\omega}(K \setminus G/ K)$. Applying the hypothesis, we have  $\mathcal{F}^{\omega}({T(f\ast_{\omega} g)}) = \mathfrak{b}\mathcal{F}^{\omega}{(f \ast_{\omega} g)} = \mathfrak{b} \mathcal{F}^{\omega}(f)\mathcal{F}^{\omega}(g) =\mathcal{F}^{\omega}(Tf)\mathcal{F}^{\omega}(g) = \mathcal{F}^{\omega}({T(f)\ast_{\omega} g})$. It follows that $T(f \ast_{\omega} g) = Tf \ast_{\omega}g$ since the weighted spherical Fourier transform is injective (Theorem\ref{injective}). Therefore, $T \in M( \mathcal{L}^{1}_{\omega}(K \setminus G/K))$.
\end{enumerate}
\end{proof}
\begin{theorem}
	Let $(G,K,\omega)$ be a weighted Gelfand pair. Let $T, T' \in M( \mathcal{L}^{1}_{\omega}(K \setminus G/K))$. Then, 
 $$\forall f,g \in \mathcal{L}^{1}_{\omega}(K \setminus G/ K),\,T{f}\ast_{\omega}T'{g} = T'{f}\ast_{\omega}T{g}.$$
\end{theorem}
\begin{proof}
	\begin{align*}
	T{f}\ast_{\omega}T'{g} & =  T(f \ast_{\omega}(T'{g})) \\
	 & = T((T'{g}) \ast_{\omega}f)\\
	 & =  T(T'(g \ast_{\omega}f))\\
	 &= T(T'(f \ast_{\omega}g))\\
	 & = T((T'{f})\ast_{\omega}g)\\
	 & =  T(g\ast_{\omega} (T'{f}))\\
	 &=  T{g} \ast_{\omega} T'{f}.\\
	 T{f}\ast_{\omega}T'{g}  &= T'{f} \ast_{\omega}T{g}.
	\end{align*}
\end{proof}

  \section{Conclusion}
In this article, the notions of weighted Gelfand pair, weighted spherical function and weighted Fourier transform on  the convolution algebra $\mathcal{L}^{1}_{\omega}(K \setminus G/ K)$   are introduced. Some of their main properties are studied. Moreover, multipliers of $\mathcal{L}^{1}_{\omega}(K \setminus G/ K)$ are characterized by the means of the weighted Fourier transform.
  \section*{Competing Interests}
The authors declare that no competing interests exist.

\centerline{}
\centerline{}
\begin{small}

 A. Y. Tissinam : Department of Mathematics, University of Lom\'e, 1 PoBox 1515 Lom\'e 1, Togo\\
 asseketis@gmail.com
 \centerline{ }
 \centerline{ }
A. Issa : Department of Mathematics, University of Lom\'e, 1 PoBox 1515 Lom\'e 1, Togo\\
 issaabudulai13@gmail.com
 \centerline{ }
 \centerline{ }
 Y.  Mensah : Department of Mathematics, University of Lom\'e, 1 PoBox 1515 Lom\'e 1, Togo\\
 and International Chair in Mathematical Physics (ICMPA)-Unesco Chair, University of Abomey-Calavi, Benin\\
mensahyaogan2@gmail.com,  ymensah@univ-lome.org
\end{small}
\end{document}